\documentclass[11pt]{article}

\usepackage[margin=1in]{geometry}
\usepackage{amsmath,amssymb,amsthm}
\usepackage{microtype}
\usepackage[hidelinks]{hyperref}
\hypersetup{
  pdfauthor={Abhishek Shankar},
  pdftitle={A Log-Free n to the 1/5 Bound for Chowla's Cosine Problem}
}

\newtheorem{theorem}{Theorem}[section]
\newtheorem{lemma}[theorem]{Lemma}
\newtheorem{proposition}[theorem]{Proposition}
\theoremstyle{remark}
\newtheorem{remark}[theorem]{Remark}

\newcommand{\T}{\mathbb T}
\newcommand{\Z}{\mathbb Z}
\newcommand{\N}{\mathbb N}
\newcommand{\1}{\mathbf 1}
\newcommand{\e}{\mathrm e}
\newcommand{\card}[1]{\lvert #1\rvert}
\newcommand{\abs}[1]{\lvert #1\rvert}

\title{A Log-Free $n^{1/5}$ Bound for Chowla's Cosine Problem}
\author{Abhishek Shankar\\[2pt]
  \normalsize Researcher at OpenProblem.ai\\[-1pt]
  \small\href{mailto:abhishek.shankar@openproblem.ai}
  {\texttt{abhishek.shankar@openproblem.ai}}\\[-1pt]
  \small ORCID:
  \href{https://orcid.org/0009-0002-6552-3622}
  {\texttt{0009-0002-6552-3622}}}
\date{Version 1, September 2026}

\begin{document}

\maketitle

\begin{abstract}
For a finite set $S$ of positive integers, put
\[
  K(S):=-\min_{x\in\T}\sum_{s\in S}\cos(2\pi sx).
\]
Bedert recently proved the uniform lower bound
$K(S)\geq \card{S}^{1/5-o(1)}$.  We remove the subpolynomial
loss and prove that
\[
  K(S)\geq c\card{S}^{1/5}
\]
for an absolute constant $c>0$.  The proof combines two estimates
from Bedert's argument with an exact averaging identity for the
asymmetric boundaries of additive intersections.  This identity
replaces the multiplicative-amplification step responsible for the
logarithmic loss.
\end{abstract}

\medskip
\noindent\textbf{2020 Mathematics Subject Classification.}
Primary 42A05; Secondary 11B30.

\smallskip
\noindent\textbf{Keywords.}
Chowla cosine problem, cosine sums, additive combinatorics.

\section{Introduction}

For a nonempty finite set $S\subset\N:=\{1,2,\ldots\}$, define
\[
  f_S(x):=\sum_{s\in S}\cos(2\pi sx),
  \qquad
  K(S):=-\min_{x\in\T}f_S(x),
\]
where $\T:=\mathbb R/\mathbb Z$.  The extremal quantity in
Chowla's cosine problem is
\[
  \mathcal K(N):=\inf_{\substack{S\subset\N\\ \card{S}=N}}K(S).
\]
Ankeny and Chowla asked whether $\mathcal K(N)\to\infty$
\cite[p.~301]{Chowla1952}, and Chowla later conjectured that the
optimal order of magnitude is $\sqrt N$
\cite[pp.~128, 130]{Chowla1965}.  In the opposite direction, a
Sidon-difference construction gives
$\mathcal K(N)\ll\sqrt N$, which remains the best known upper bound;
see \cite[Section~1]{Bedert2026}.  Thus Chowla conjectured that this
construction has the correct order of magnitude.

The first proof that $\mathcal K(N)\to\infty$ follows from Cohen's
work on the Littlewood $L^1$ conjecture \cite{Cohen1960}, as observed
by S.~Uchiyama and M.~Uchiyama \cite{UchiyamaUchiyama1960}.  Roth
subsequently proved
\[
  \mathcal K(N)\gg
  \left(\frac{\log N}{\log\log N}\right)^{1/2}
\]
\cite{Roth1973}.  The resolution of the Littlewood $L^1$ conjecture
by McGehee, Pigno and Smith \cite{McGeheePignoSmith1981}, and
independently by Konyagin \cite{Konyagin1981}, yields the stronger
logarithmic bound $\mathcal K(N)\gg\log N$.  Bourgain was the first to
cross the logarithmic barrier in his 1984 Orsay preprint and subsequent
published work, proving
$\mathcal K(N)>\exp((\log N)^\varepsilon)$ for some
$\varepsilon>0$ \cite{Bourgain1984,Bourgain1986}.  Ruzsa refined Bourgain's
method to obtain
$\mathcal K(N)\geq\exp(c\sqrt{\log N})$ for an absolute constant
$c>0$ \cite{Ruzsa2004}.  Sanders also proved a polynomial estimate
in the special case in which every frequency is $O(N)$
\cite{Sanders2010}.

The first general polynomial lower bounds were obtained independently
and almost simultaneously.  Jin, Milojevi\'c, Tomon and Zhang proved
$\mathcal K(N)\geq N^{1/10-o(1)}$
\cite{JinMilojevicTomonZhang2025}, while Bedert's version~1 gave
$\mathcal K(N)\gg N^{1/12}$ \cite[version~1]{Bedert2026}.
Bedert raised the exponent to $1/7-o(1)$ in version~2 and to
$1/5-o(1)$ in version~3.  More explicitly, combining the
multiplicative-amplification estimate in
\cite[Lemma~7.4 in version~3]{Bedert2026} with the asymmetric-boundary
estimate in \cite[Proposition~7.3 in version~3]{Bedert2026} gives
\[
  \mathcal K(N)\gg \left(\frac{N}{(\log N)^4}\right)^{1/5}
  =N^{1/5-o(1)};
\]
see also \cite[Theorem~1.1 in version~3]{Bedert2026}.  To our
knowledge, the following result removes this logarithmic loss.

\begin{theorem}\label{thm:main}
There is an absolute constant $c>0$ such that every nonempty finite
set $S\subset\N$ satisfies
\[
  \min_{x\in\T}\sum_{s\in S}\cos(2\pi sx)
  \leq -c\card{S}^{1/5}.
\]
Equivalently, $\mathcal K(N)\gg N^{1/5}$.
\end{theorem}

The proof is short.  We retain two ingredients from Bedert's work:
a Roth-type lower bound for the number of additive triples and a
pointwise upper bound for certain asymmetric boundary sets.  The new
ingredient in our argument is the exact identity in
Lemma~\ref{lem:total-boundary},
which converts the global triple count directly into a large boundary.
It thereby makes the multiplicative-amplification argument in
\cite[Lemma~7.4]{Bedert2026} unnecessary.

\section{Two inputs from Bedert's argument}

Write $e(x):=\e^{2\pi i x}$.  If
$A\subset\Z$ is finite, set
\[
  F_A(x):=\sum_{a\in A}e(ax).
\]
When $A=-A$, the polynomial $F_A$ is real-valued.  We record the
precise forms of the two results from \cite{Bedert2026} that we use.
Bedert writes $\widehat{1_A}$ for our $F_A$; the conventions agree,
since both use $e(x)=e^{2\pi ix}$ on $\T=\mathbb R/\mathbb Z$.

\begin{proposition}[Roth--Bedert]\label{prop:triple-input}
Let $A=-A\subset\Z\setminus\{0\}$ be finite and nonempty, let
$K>0$, and suppose that $F_A(x)+K\geq0$ for every $x\in\T$.
If $E\subseteq A$ and
$\card{E}\geq2K^2$, then
\[
  \#\{(u,v)\in E^2:u-v\in A\}
  \geq \frac{\card{E}^2}{2K}.
\]
\end{proposition}

This is the Roth-type difference-density estimate in
\cite[Lemma~4.1 in version~3]{Bedert2026}, where it is attributed to
Roth.  The second input is Bedert's estimate for the asymmetric part
of $A\cap(A+t)$, stated as
\cite[Proposition~7.3 in version~3]{Bedert2026}.

\begin{proposition}[Bedert]\label{prop:boundary-input}
There is an absolute constant $C_0>0$ with the following property.
Let $A=-A\subset\Z\setminus\{0\}$ be finite and nonempty, let
$K>0$, and suppose that $F_A(x)+K\geq0$ for every $x\in\T$.
For $t\in\Z$, define
\[
  A_t:=A\cap(A+t),
  \qquad
  B_t:=A_t\setminus(-A_t).
\]
Then
\[
  \card{B_t}\leq C_0K^4
\]
for every $t\neq0$.
\end{proposition}

All constants below are absolute.  No quantitative information about
$C_0$ will be needed.
All integrals over $\T$ are taken with respect to normalized Haar
measure.
The constant $C_0$, and hence the constant in Theorem~\ref{thm:main},
is effectively computable by tracking the constants in Bedert's
proof; we make no attempt to optimize it.

\section{The boundary-averaging identity}

We next give the new counting observation.  The strength needed for
Theorem~\ref{thm:main} is the inequality
\eqref{eq:boundary-lower}; the more precise identity
\eqref{eq:exact-boundary} makes the multiplicities transparent.

\begin{lemma}[Total-boundary identity]\label{lem:total-boundary}
Let $A=-A\subset\Z\setminus\{0\}$ be finite.  For $t\in\Z$, define
\[
  A_t:=A\cap(A+t),
  \qquad
  B_t:=A_t\setminus(-A_t),
\]
and put
\[
  T:=\#\{(a,t)\in A^2:a-t\in A\}.
\]
Then
\begin{equation}\label{eq:boundary-lower}
  \sum_{t\in A}\card{B_t}\geq\frac{T}{3}.
\end{equation}
More precisely, if $P:=A\cap\N$ and
\begin{align*}
  Q&:=\#\{(u,v)\in P^2:u+v\in P\},\\
  I&:=\#\{(u,v)\in P^2:u-v\in A\ \text{and}\ u+v\in A\},
\end{align*}
then $0\leq I\leq Q$ and
\begin{equation}\label{eq:exact-boundary}
  T=6Q,
  \qquad
  \sum_{t\in A}\card{B_t}=6Q-4I.
\end{equation}
\end{lemma}

\begin{proof}
Since $A=-A$, we have
\[
  -A_t=A\cap(A-t)=A_{-t}.
\]
Consequently,
\begin{equation}\label{eq:Bt-pointwise}
  B_t=\{a\in A:a-t\in A,\ a+t\notin A\}.
\end{equation}

For $u,v\in P$, define
\[
  d_{u,v}:=\1_A(u-v),
  \qquad
  p_{u,v}:=\1_A(u+v),
  \qquad
  C_{u,v}:=d_{u,v}-p_{u,v}.
\]
Both $d_{u,v}$ and $p_{u,v}$ belong to $\{0,1\}$.  For
$a,t\in A$, set
\[
  \beta(a,t):=\1_A(a-t)\bigl(1-\1_A(a+t)\bigr).
\]
By \eqref{eq:Bt-pointwise},
\[
  \sum_{t\in A}\card{B_t}
  =\sum_{(a,t)\in A^2}\beta(a,t).
\]
Fix $u,v\in P$.  The four ordered pairs $(a,t)\in A^2$ with
$\abs{a}=u$ and $\abs{t}=v$ contribute
\begin{align*}
  \beta(u,v)=\beta(-u,-v)&=d_{u,v}(1-p_{u,v}),\\
  \beta(u,-v)=\beta(-u,v)&=p_{u,v}(1-d_{u,v}).
\end{align*}
Their total contribution is
\[
  2d_{u,v}(1-p_{u,v})+2p_{u,v}(1-d_{u,v})
  =2(d_{u,v}-p_{u,v})^2.
\]
Every ordered pair in $A^2$ lies in exactly one such sign class,
including when $u=v$.  Hence
\begin{equation}\label{eq:boundary-frobenius}
  \sum_{t\in A}\card{B_t}
  =2\sum_{u,v\in P}C_{u,v}^2.
\end{equation}

Let
\[
  D:=\#\{(u,v)\in P^2:u-v\in A\}.
\]
The diagonal contributes nothing, because $0\notin A$.  The map
\[
  (u,v)\longmapsto(v,u-v)
\]
is a bijection from the pairs counted by $D$ with $u>v$ to the
pairs counted by $Q$; its inverse is
$(x,y)\mapsto(x+y,x)$.  Swapping $u$ and $v$ handles the pairs
with $u<v$, and therefore
\begin{equation}\label{eq:D-equals-2Q}
  D=2Q.
\end{equation}
Moreover, because $u+v>0$, the condition $u+v\in A$ is equivalent
to $u+v\in P$, and thus $0\leq I\leq Q$.  Expanding the square in
\eqref{eq:boundary-frobenius} gives
\begin{align}
  \sum_{u,v\in P}C_{u,v}^2
  &=\sum_{u,v\in P}
    \bigl(d_{u,v}+p_{u,v}-2d_{u,v}p_{u,v}\bigr)\notag\\
  &=D+Q-2I
   =3Q-2I.\label{eq:frobenius-count}
\end{align}
Together, \eqref{eq:boundary-frobenius} and
\eqref{eq:frobenius-count} show that
\begin{equation}\label{eq:boundary-QI}
  \sum_{t\in A}\card{B_t}=6Q-4I.
\end{equation}

It remains to identify $T$.  The map
\[
  (a,t)\longmapsto(t,a-t,-a)
\]
is a bijection from the pairs counted by $T$ to the zero-sum triples
\[
  \mathcal Z:=\{(z_1,z_2,z_3)\in A^3:
  z_1+z_2+z_3=0\}.
\]
Because $0\notin A$, each triple in $\mathcal Z$ has two positive
entries and one negative entry, or two negative entries and one
positive entry.  Fixing one of the six possible sign patterns, the
absolute values of the two entries with the same sign, in their
coordinate order, form a pair counted by $Q$.  Conversely, every
pair counted by $Q$ determines one zero-sum triple of each fixed sign
pattern.  Thus
\begin{equation}\label{eq:T-equals-6Q}
  T=\card{\mathcal Z}=6Q.
\end{equation}
This also covers $u=v$: coordinate order fixes the multiplicity, so
no additional factor occurs.

Equations \eqref{eq:boundary-QI} and \eqref{eq:T-equals-6Q} prove
\eqref{eq:exact-boundary}.  Finally, since $I\leq Q$,
\[
  \sum_{t\in A}\card{B_t}
  =6Q-4I\geq2Q=\frac{T}{3},
\]
which is \eqref{eq:boundary-lower}.
\end{proof}

\begin{remark}\label{rem:fixed-column}
The proof also gives the fixed-column identity
\[
  \card{B_t}=\sum_{u\in P}
  \bigl(\1_A(u-t)-\1_A(u+t)\bigr)^2
  \qquad(t\in P).
\]
Thus the total-boundary estimate is not an application of an
inequality that loses logarithms; it is an exact count followed only
by the elementary bound $I\leq Q$.
\end{remark}

\begin{remark}[Comparison with Bedert]\label{rem:comparison}
For comparison, \cite[Lemma~7.4 in version~3]{Bedert2026} produces a
nonzero $t$ with
\[
  \card{B_t}\gg\frac{n}{K(\log n)^4}
\]
by amplifying a single large intersection $A\cap(A+t_0)$ along the
multiplicative semigroup generated by the primes up to
$\asymp(\log n)^2$.  When $n\geq2K^2$, combining
Proposition~\ref{prop:triple-input} with
Lemma~\ref{lem:total-boundary} instead gives
$\card{B_t}\geq n/(6K)$ for some $t\in A$.  The factor
$(\log n)^4$ is the only logarithmic loss in
\cite[Section~7 of version~3]{Bedert2026}.
\end{remark}

\begin{remark}[Sharpness of the averaging constant]
The factor $1/3$ in \eqref{eq:boundary-lower} is asymptotically
sharp.  If $P=\{1,\ldots,m\}$ and $A=P\cup(-P)$, then
$I/Q\to1$, and consequently
\[
  \frac{\sum_{t\in A}\card{B_t}}{T}
  =1-\frac{2I}{3Q}\longrightarrow\frac13.
\]
For example, when $m=79$ one has $I/Q=78/79$ and
$\sum_{t\in A}\card{B_t}=(81/79)(T/3)$.
\end{remark}

\section{Proof of the main theorem}

We first prove a symmetric exponential-sum formulation.

\begin{theorem}\label{thm:symmetric}
There is an absolute constant $c_0>0$ such that, for every nonempty
finite symmetric set $A=-A\subset\Z\setminus\{0\}$,
\[
  -\min_{x\in\T}F_A(x)\geq c_0\card{A}^{1/5}.
\]
\end{theorem}

\begin{proof}
Put $n:=\card{A}$ and
\[
  K:=-\min_{x\in\T}F_A(x).
\]
Then $F_A+K\geq0$ pointwise.  We first note that $K\geq1$.  Indeed,
if $a_0\in A$, then the nonnegative function $G:=F_A+K$ satisfies
\[
  1=\abs{\int_{\T}G(x)e(-a_0x)\,dx}
  \leq\int_{\T}G(x)\,dx=K.
\]

If $n<2K^2$, then $K\geq1$ gives
\[
  n<2K^2\leq2K^5.
\]
We may therefore assume that $n\geq2K^2$.  Applying
Proposition~\ref{prop:triple-input} with $E=A$ yields
\[
  T:=\#\{(a,t)\in A^2:a-t\in A\}
  \geq\frac{n^2}{2K}.
\]
Lemma~\ref{lem:total-boundary} now gives
\[
  \sum_{t\in A}\card{B_t}
  \geq\frac{T}{3}
  \geq\frac{n^2}{6K}.
\]
Averaging over the $n$ elements of $A$, there is a $t\in A$ such
that
\begin{equation}\label{eq:large-boundary}
  \card{B_t}\geq\frac{n}{6K}.
\end{equation}
Since $0\notin A$, this $t$ is nonzero.  Proposition
\ref{prop:boundary-input} and \eqref{eq:large-boundary} therefore
imply
\[
  \frac{n}{6K}\leq C_0K^4,
\]
and hence $n\leq6C_0K^5$.  Combining the two cases proves
\[
  K\geq c_0n^{1/5},
  \qquad
  c_0:=\max\{2,6C_0\}^{-1/5}>0.
\]
\end{proof}

\begin{proof}[Proof of Theorem~\ref{thm:main}]
Let $S\subset\N$ be nonempty, put $N:=\card{S}$, and symmetrize:
\[
  A:=S\cup(-S).
\]
The union is disjoint, so $\card{A}=2N$, and
\[
  F_A(x)
  =\sum_{s\in S}\bigl(e(sx)+e(-sx)\bigr)
  =2\sum_{s\in S}\cos(2\pi sx).
\]
If $M:=K(S)$, then $-\min_{x\in\T}F_A(x)=2M$.  Theorem
\ref{thm:symmetric} gives
\[
  2M\geq c_0(2N)^{1/5},
\]
and consequently
\[
  M\geq c_0\,2^{-4/5}N^{1/5}.
\]
This proves the result with $c=c_0\,2^{-4/5}$.
\end{proof}

\section{Concluding remarks}

Theorem~\ref{thm:main} removes the logarithmic loss from the
$1/5$-exponent bound in \cite{Bedert2026}, but it does not improve the
exponent itself.  A substantial gap therefore remains between the
present estimate $\mathcal K(N)\gg N^{1/5}$ and Chowla's conjectured
square-root scale $\mathcal K(N)\gg N^{1/2}$.

The analytic input in this note is entirely contained in
Propositions~\ref{prop:triple-input} and
\ref{prop:boundary-input}.  The new contribution is the exact
total-boundary identity of Lemma~\ref{lem:total-boundary} and its use
to pass from the global additive-triple count to a single large
boundary without logarithmic loss.  The identity controls the sum of
the boundary sizes, and the present proof uses it only to extract one
large $B_t$.  Improving the exponent past $1/5$ by this route would
therefore require either a stronger exponent than $K^4$ in
Proposition~\ref{prop:boundary-input}, or additional information
showing that many boundary sets are simultaneously large.

\section*{Statements and declarations}

\paragraph{Acknowledgements.}
The author thanks Benjamin Bedert for making the arguments on which
this note builds publicly available.

\paragraph{Author responsibility and tool use.}
Abhishek Shankar is the sole named author and accepts responsibility
for every claim.  AI-assisted tools were used for exploratory
computations, drafting assistance, and automated sanity checks.  The
author is solely responsible for all mathematical claims and conclusions.

\paragraph{Data and code availability.}
No datasets were generated or analyzed.  A short Python program
verifying the new finite counting identity for all nonempty subsets
of $\{1,\ldots,15\}$ is included with the ancillary source files.
This computation is a sanity check and is not used in the proof.

\paragraph{Funding and competing interests.}
The author received no specific funding for this work and declares
no competing interests.

\begingroup
\small
\bibliographystyle{alpha}
\bibliography{references}
\endgroup

\end{document}